\documentclass[12pt, twoside]{article}
\usepackage{amsmath,amsthm,amssymb}
\usepackage{times}
\usepackage{enumerate}
\usepackage{bm}
\usepackage[all]{xy}
\usepackage{mathrsfs}
\usepackage{amscd}
\usepackage{color}
\def\titlerunning#1{\gdef\titrun{#1}}
\makeatletter
\def\author#1{\gdef\autrun{\def\and{\unskip, }#1}\gdef\@author{#1}}
\def\address#1{{\def\and{\\\hspace*{18pt}}\renewcommand{\thefootnote}{}%
		\footnote {#1}}%
	\markboth{\autrun}{\titrun}}
\makeatother
\def\email#1{e-mail: #1}

\newtheorem{theorem}{Theorem}[section]
\newtheorem{corollary}[theorem]{Corollary}
\newtheorem{lemma}[theorem]{Lemma}
\newtheorem{proposition}[theorem]{Proposition}
\theoremstyle{definition}
\newtheorem{definition}[theorem]{Definition}
\newtheorem{remark}[theorem]{Remark}
\newtheorem{example}[theorem]{Example}
\numberwithin{equation}{section}

\xyoption{all}

\def \a {\alpha }
\def \b {\beta}

\def \de {\delta}
\def \De {\Delta}
\def \la {\lambda}
\def \La {\Lambda}
\def\w {\omega}
\def\Om{\Omega}

\def\pa{\partial}
\def\na {\nabla}
\def\Ga{\Gamma}

\begin{document}
	\baselineskip=17pt
	
	\titlerunning{Global perturbation potential function on complete special holonomy manifolds }
	\title{Global perturbation potential function on complete special holonomy manifolds}
	
	\author{Teng Huang}
	
	\date{}
	
	\maketitle
	
	\address{T. Huang: School of Mathematical Sciences, University of Science and Technology of China; CAS Key Laboratory of Wu Wen-Tsun Mathematics,  University of Science and Technology of China; Hefei, Anhui 230026, PR China; \email{htmath@ustc.edu.cn;htustc@gmail.com}}
	
	\begin{abstract}
In this article,  we introduce and study the notion of a complete special holonomy manifold $(X,\w)$ which is given by a global perturbation potential function, i.e., there is a function $f$ on $X$ such that $\w'=\w-\mathcal{L}_{\na f}\w$ is sufficiently small in $L^{\infty}$-norm. We establish some vanishing theorems on the $L^{2}$ harmonic forms under some conditions on the global perturbation potential function. 
	\end{abstract}
	\section{Introduction}
	Let $X$ be a smooth Riemannian manifold equipped with a differential form $\w$. This form is called parallel if $\w$ is preserved by the Levi-Civita connection: $\na\w=0$. This identity gives a powerful restriction on the holonomy group $\rm{Hol}(X)$. The structure of $\rm{Hol}(X)$ and its relation to the geometry of a manifold is one of the main subjects of Riemannian geometry of the last 50 years. In K\"{a}hler geometry the parallel forms are the K\"{a}hler form and its powers. The algebraic geometers obtained many topological and geometric results on studying the corresponding algebraic structure. In $G_{2}$- or $Spin(7)$-manifold the parallel form is the $G_{2}$- or $Spin(7)$-structure. In \cite{Ver}, Verbitsky had  generalized some of these results on K\"{a}hler manifolds to other manifolds with a parallel form, especially the parallel $G_{2}$-manifolds. The results obtained in \cite{Ver} can be summarized as K\"{a}hler identities for $G_{2}$-manifolds.
	
	The theory of $G_{2}$-manifolds is one of the places where mathematics and physics interact most strongly \cite{KL,LL}. In string theory, $G_{2}$-manifolds are expected to play the same role as Calabi-Yau manifolds in the usual A- and B-models of type-II string theories. There are many results on the construction of $G_{2}$-manifolds \cite{Bry,Joy1,Joy2,Kov}. In \cite{CHJP}, Corti-Haskins-Nordstr\"{o}m-Pacini  constructed many new topological types of compact $G_{2}$-manifolds by applying the twisted connected sum to asymptotically Calabi-Yau 3-folds of semi-Fano type studied in \cite{CHJP2}. Joyce-Karigiannis also given a new construction of compact Riemannian $7$-manifolds with holonomy $G_{2}$ (See \cite{JK}). Hitchin constructed a geometric flow \cite{Hit2} which physicists called Hichin's flow. This has turned out to be extremely important in string physics.
	
	The study of $L^{2}$ harmonic forms on a complete special holonomy manifold is a very interesting and important subject; it also has numerous applications in the field of Mathematical Physics, see for example \cite{Hit}. In K\"{a}hler geometry (holonomy $U(n)$) the parallel forms are the K\"{a}hler form $\w$ and its powers. Studying the corresponding
	algebraic structures, the algebraic geometers amassed an amazing wealth of topological and geometric information. There are many vanishing results on K\"{a}hler geometry. The first general result in the non-compact case is due to Donnelly-Fefferman \cite{DF}. If $X$ is a strongly pseudoconvex domain in $\mathbb{C}^{n}$, they showed in \cite{DF} that $\mathcal{H}^{p,q}_{(2)}(X)=0$, $p+q\neq n$, if $\w$ is the Bergman metric.  In \cite{Gro}, Gromov introduced the notion of K\"{a}hler hyperbolicity and established the vanishing of $\mathcal{H}^{p,q}_{(2)}(X)$, outside the middle dimension, for any $(X,\w)$ which is K\"{a}hler hyperbolic and which covers a compact manifold. In \cite{CX,JZ}, Cao-Frederico and Jost-Zuo proved that $\mathcal{H}^{p,q}_{(2)}(X)=0$, $p+q\neq n$, if $\w=d\a$ with $\|\a\|_{L^{\infty}(X)}$ growing slower than the Riemannian distance associated to $\w$. Assume that $\w$ is given by a global potential function, i.e., there is a $\la\in C^{2}(X)$ such that 
	$$\w=\text{i}\pa\bar{\pa}\la=\frac{1}{2}dd^{C}\la,$$
	where $d^{C}:=[\mathcal{L}_{\w},d^{\ast}]=-\rm{i}(\pa-\bar{\pa})$. In \cite{McN1,McN2}, McNeal proved two vanishing theorems on $\mathcal{H}^{p,q}_{(2)}(X)$ when $p+q\neq n$, under some growth assumptions on the global potential function $f$.
	
	For the case of complete $G_{2}$- or $Spin(7)$-manifold $X$, it well-known that $\mathcal{H}^{i}_{(2)}(X)=0$, $i=0,1$, since $X$ is Ricci-flat. The author in \cite{Hua} proved that $\mathcal{H}^{2}_{(2)}(X)=0$ if the structure form $\w=d\a$ with $\|\a\|_{L^{\infty}(X)}$ grows slower than the Riemannian distance associated to the metric $g_{\w}$ induced by $\w$.
	
	We define a $\phi$-plurisubharmonic function on a calibrated manifold $(X,\phi)$ where $\deg(\phi)=p$. Harvey and Lawson \cite{HL} introduced a second order differential operator $\mathcal{H}^{\phi}: C^{\infty}(X)\rightarrow \La^{p}(X)$, the $\phi$-Hessian given by
	$$\mathcal{H}^{\phi}(f)=\la_{\phi}(\text{Hess}f),$$
	where $\text{Hess}f$ is the Riemannian Hessian of $f$ and $\la_{\phi} :\text {End(TX)}\rightarrow\La^{p}(X)$ is the bundle map given by $\la_{\phi}A=D_{A^{\ast}}(\phi)$ where $D_{A^{\ast}}:\La^{p}T^{\ast}X\rightarrow\La^{p}T^{\ast}X$ is the natural extension of $A^{\ast}: T^{\ast}X\rightarrow T^{\ast}X$ as a derivation. When the calibration $\phi$ is parallel there is a natural factorization
		$$\mathcal{H}^{\phi}=dd^{\phi},$$
		where $d$ is the de Rham differential and $d^{\phi}:C^{\infty}(X)\rightarrow\La^{p-1}(X)$ is given by
		$$d^{\phi}f=i_{\na f}\phi.$$
Inspired by K\"{a}hler geometry, a parallel differential $k$-form $\w$ on a complete manifold $X$ may be given by a function $f$, i.e., there is a $f\in C^{2}(X)$ such that 
$$\w=\mathcal{L}_{\na f}\w.$$
where  we denote by $\mathcal{L}_{\na f}$ the Lie derivative of the vector field $\na f$ which is the metric dual of the $1$-form $df$.
\begin{remark}
Suppose that $(X,\w)$ is a complete manifold with holonomy $G_{2}$ or $Spin(7)$, and $\w$ is the structure form and there is a smooth function $f$ on $X$ such that the Lie derivative $\mathcal{L}_{\nabla f}\w=\w$ on $X$. Then the only possibility for $(X,\w)$ is $\mathbb{R}^{7}$ or $\mathbb{R}^{8}$ with the Euclidean $G_{2}$ or $Spin(7)$ structure. Since on a $G_{2}$ and $Spin(7)$-manifold, the structure form $\w$ determines a metric $g$, $\mathcal{L}_{\nabla f}\w=\w$ implies that  $\mathcal{L}_{\nabla f}g =g$. Following the flow of $\nabla f$ backwards, one can see that it shrinks the manifold down to a point in finite distance (though infinite time). As $(X,\w)$ is complete, this must be a nonsingular point, so $(X,\w)$ must be Euclidean $\mathbb{R}^{7}$ or $\mathbb{R}^{8}$. 
\end{remark}
In this article, we will study the case where the $k$-parallel form $\w$ given by a global perturbation potential $f$, i.e, there is a function $f\in C^{2}(X)$ such that $$\w':=\w-\mathcal{L}_{\na f}\w=\w-(-1)^{\tilde{C}}dd_{C}f$$
is sufficiently small  in $L^{\infty}$-norm. One also can see Proposition \ref{P3.1} and Definition \ref{D2}. The main purpose of this article is to prove some vanishing results of the harmonic forms on $X$ if $f$ is a convex function.
 \begin{example}\label{Ex1}(i) Let $(X,\w,\Om)$ be a nearly K\"{a}hler $6$-fold \cite{Ver2005,Ver2011}. There is a $(3,0)$-form $\Om$ with $|\Om|=1$, and
 	$$d\w=3\la Re\Om,\ dIm\Om=-2\la\w^{2},$$
 	where $\la$ is a non-zero real constant. For simplicity, we choose $\la=1$. Denote by $C(X)$ the Riemannian cone of $(X,g)$. The Riemannian cone $\big{(}C(X),dr^{2}+r^{2}g\big{)}$ is a $G_{2}$-manifold with torsion-free $G_{2}$-structure $\phi$ defined by
 	$$\phi:=r^{2}\w\wedge dr+r^{3}Re\Om.$$
 	We denote $f=\frac{1}{6}r^{2}$, thus $\na f=\frac{1}{3}r\frac{\pa}{\pa r}$. In a direct calculation, 
 	$$\mathcal{L}_{\na f}\phi=\text{di}_{\na f}\phi=d(\frac{1}{3}r^{3}\w)=\phi.$$
 	Therefore the Riemaniann cone $C(X)$ is given by a global potential $\frac{1}{6}r^{2}$.\\ 
 	(ii) Let $(X,\phi)$ be a nearly parallel $G_{2}$-manifold \cite{Iva}. There is a $3$-form $\phi$ with $|\phi|^{2}=7$ such that
 	$$d\phi=4\ast\phi.$$
 	Then the Riemannian cone $\big{(}C(X),dr^{2}+r^{2}g\big{)}$ is a $Spin(7)$-manifold with  $Spin(7)$-structure $\Phi$ defined by
 	$$\Phi:=r^{3}dr\wedge\phi+r^{4}\ast\phi.$$
 	We denote $f=\frac{1}{8}r^{2}$, thus $\na f=\frac{1}{4}r\frac{\pa}{\pa r}$. In a direct calculation, 
 	$$\mathcal{L}_{\na f}\Phi=\text{di}_{\na f}\Phi=d(\frac{1}{4}r^{4}\phi)=\Phi.$$
 	Therefore the Riemaniann cone $C(X)$ given by a global potential $\frac{1}{8}r^{2}$. The Riemannian cones $(C(X),dr^{2}+g)$ are not complete manifolds. 
 	
 	Karigiannis in \cite[Definition 2.33]{Kar2} defined  an asymptotically conical $G_{2}$ manifold with cone $C$ and rate $\nu<0$ if all of the following holds:\\
 	(a) The manifold $N$ is a $G_{2}$-manifold with torsion-free $G_{2}$-structure $\phi_{N}$ and metric $g_{N}$.\\
 	(b) There is a $G_{2}$-cone $(C,\phi_{C},g_{C})$ with link $\Sigma$.\\
 	(c) There is a compact subset $L\subset N$.\\
 	(d) There is an $R>1$, and a smooth function $h:(R,\infty)\times\Sigma\rightarrow N$ that is a diffeomorphism of $(R,\infty)\times\Sigma$ onto $N\backslash L$.\\
 	(e) The pull back $h^{\ast}(\phi_{N})$ is a torsion-free $G_{2}$-structure on the subset $(R,\infty)\times\sigma$ of $C$. We require that this approach the torsion-free $G_{2}$-structure $\phi_{C}$ in a $C^{\infty}$, with rate $\nu<0$. This means that
 	$$|\na_{C}^{j}(h^{\ast}(\phi_{N})-\phi_{C})|_{g_{C}}=O({r^{\nu-j}}),\ \forall j\geq0,$$
 	in $(R,\infty)\times\Sigma$.
 	
If $h$ is identity map, then $N=L\cup(R,\infty)\times \Sigma$. Therefore, $\phi_{N}=\phi_{C}:=\mathcal{L}_{\frac{1}{3}r\frac{\pa}{\pa r}}\phi_{N}$ on $(R,\infty)\times\Sigma$. We can choose a smooth positive function $f$ such that $f=\frac{1}{6}r^{2}$ on $(R,\infty)\times\Sigma$.  Then there is a $3$-form $\phi_{N}'$ such that $\phi_{N}=\mathcal{L}_{\na f}\phi_{N}+\phi_{N}'$. Since $L$ is compact, $|\na f|$ has a upper bound on $L$. The function $f$ satisfies the convexity condition, see Definition \ref{D1}. One can also obviously consider asymptotically conical $Spin(7)$-manifolds.
 \end{example} 
At first, we give an estimate on  $L^{2}$-harmonic form as follows.
\begin{theorem}\label{T1}
Let  $(X,\w)$ be a complete Riemannian manifold equipped with a non-zero  parallel differential $k$-form $\w$.  Suppose that there exist a smooth exhaustion function $\la\geq1$ on $X$ and a $k$-form $\w'$ on $X$ such that  $\w=(-1)^{\tilde{C}}dd_{C}f+\w'$.  Also assume that the function $f$ satisfies the convexity condition on $X$, i.e., for some $A,B\geq0$, $|df|^{2}\leq A+Bf$. Then for any $h\in\mathcal{H}^{p}_{(2)}(X)$, we have  
$$\|\w\wedge h\|_{L^{2}(X)}\leq\|\w'\wedge h\|_{L^{2}(X)}.$$	
\end{theorem}
We call the map on $\Om^{k}(X)$, 
\begin{equation*}
\begin{split}
L_{\w}:&\Om^{p}(X)\rightarrow\Om^{k+p}(X)\\
&\a\mapsto\w\wedge\a\\
\end{split}
\end{equation*} 
the general Lefschetz map. 
\begin{remark}
(1) If $(X,\w)$ is a K\"{a}hler manifold with real dimension $2n$, $\w$ is the K\"{a}her form, then the map $L_{\w}$ is bijective for all $k<n$ \cite{Wel}.\\
(2) If $(X,\w)$ is a $G_{2}$ or $Spin(7)$-manifold, $\w$ is the structure form, then the map $L_{\w}$ is bijective for $k=0,1,2$ (see Lemma \ref{L2}, \ref{L3}).
\end{remark}
\begin{corollary}\label{C2}
Let  $(X,\w)$ be a complete Riemannian manifold equipped with a non-zero  parallel differential $k$-form $\w$. Suppose that there exist a smooth exhaustion function $\la\geq1$ on $X$ and a $k$-form $\w'$ on $X$ such that  $\w=(-1)^{\tilde{C}}dd_{C}f+\w'$. Also assume that the function $f$ satisfies the convexity condition on $X$, i.e., for some $A,B\geq0$, $|df|^{2}\leq A+Bf$ and the $k$-form $\w'$ obeys
$$\|\w'\|_{L^{\infty}(X)}\leq\varepsilon,$$
If $\varepsilon=\varepsilon(n)\in(0,1]$ is sufficiently small, then\\
(1) if $X$ is a K\"{a}hler manifold, then for $k\neq n$, $$\mathcal{H}^{k}_{(2)}(X)=\{0\}.$$
(2) if $X$ is a $G_{2}$ or $Spin(7)$-manifold, then for $k=0,1,2$,
	$$\mathcal{H}^{k}_{(2)}(X)=\{0\}.$$

\end{corollary}
 A differential form $\a$ on a complete non-compact Riemannian manifold $(X,g)$ is called $d$(sublinear) if there exist a differential form $\b$ and a number $c>0$ such that
$\a=d\b$ and
$$|\a(x)|_{g}\leq c\ and\ |\b(x)|_{g}\leq c(1+\rho(x,x_{0})),$$
where $\rho(x,x_{0})$ stands for the Riemannian distance between $x$ and a base point $x_{0}$ with respect to $g$. One can see that $\w'$ is closed on $X$. We then prove that
\begin{theorem}\label{T2}
	Let  $(X,\w)$ be a complete Riemannian manifold equipped with a non-zero  parallel differential $k$-form $\w$. Suppose that there exist a smooth exhaustion function $\la\geq1$ on $X$ and a $k$-form $\w$ on $X$ such that  $\w=(-1)^{\tilde{C}}dd_{C}f+\w'$ on $X$. Also assume that the function $f$ satisfies the convexity condition on $X$ and $\w'$ is $d$(sublinear). Then for any $h\in\mathcal{H}^{p}_{(2)}(X)$, we have  
	$$\w\wedge h=0.$$	
\end{theorem}
We could prove an other vanishing result if the $k$-form $\w'$ is $d$(sublinear). In this condition, the form $\w'$ may be infinite in  $L^{\infty}$-norm which is slightly different to the hypotheses in Corollary \ref{C2}.
\begin{corollary}\label{C3}
Let  $(X,\w)$ be a complete Riemannian manifold equipped with a non-zero  parallel differential $k$-form $\w$. Suppose that there exist a smooth exhaustion function $\la\geq1$ on $X$ and a $k$-form $\w'$ on $X$ such that  $\w=(-1)^{\tilde{C}}dd_{C}f+\w'$. Also assume that the function $f$ satisfies the convexity condition on $X$ and the $k$-form $\w'$ is $d$(sublinear). Then,\\
(1) if $X$ is a K\"{a}hler manifold, then for $k\neq n$, $$\mathcal{H}^{k}_{(2)}(X)=\{0\}.$$
(2) if $X$ is a $G_{2}$ or $Spin(7)$-manifold, then for $k=0,1,2$,
$$\mathcal{H}^{k}_{(2)}(X)=\{0\}.$$
\end{corollary}
Suppose that $X$ is a $G_{2}$ or $Spin(7)$-manifold. If the gradient of $f$ less than $f$, i.e., $|df|^{2}\leq A+Bf$, where $A,B\geq0$ are constants; and $B$, $\w'$ are small enough, then we obtain a lower bound on $(\De u,u)$ for $u\in\Om^{k}_{(2)}(X)$, $k=0,1,2$. 
\begin{theorem}\label{T3}
	Let  $(X,\w)$ be a complete $G_{2}$- (or $Spin(7)$-) manifold. Let $k=0,1,2$. Suppose that there exist a smooth  function $\la\geq1$ on $X$ and a $k$-form $\w$ on $X$ such that  $\w=(-1)^{\tilde{C}}dd_{C}f+\w'$ on $X$. Also assume that the function $f$ satisfies the convexity condition on $X$, i.e., for some $A,B\geq0$, $|df|^{2}\leq A+Bf$. Then there is a positive constant $\de\in(0,1]$ with the following significance. If $B\leq\de$ and $|\w'|\leq\de$, there exist constants $m$, $M$ depending only on universal constants and the  constants $A,B$ such that
	\begin{equation}\label{E60}
	m\int_{X}\frac{1}{f+M}|u|^{2}\leq (\|du\|^{2}+\|d^{\ast}u\|^{2}),\ \forall u\in \La^{k}_{0}(X),
	\end{equation}
	In particular, 
	$$\mathcal{H}^{k}_{(2)}(X)=0.$$
\end{theorem}
 As we derive estimates in our article, there will be many constants which appear. Sometimes we will take care to bound the size of these constants, but we will also use the following notation whenever the value of the constants is unimportant. We write $\a\lesssim\b$ to mean that $\a\leq C\b$ for some positive constant $C$ independent of certain parameters on which $\a$ and $\b$ depend. The parameters on which $C$ is independent will be clear or specified at each occurrence. We also use $\b\lesssim\a$ and $\a\approx\b$ analogously.
\section{Preliminaries}
\subsection{$L^{2}$-harmonic forms}
We recall some basic  facts on $L^{2}$ harmonic forms \cite{Car1,Car2}.  Let $M$ be a smooth manifold of dimension $n$, let $\La^{k}(M)$ and $\La^{k}_{0}(M)$ denote the smooth $k$-forms on $M$ and the smooth $k$-forms  with compact support on $M$, respectively. We assume now that $M$ is endowed with a Riemannian metric $g$. Let $\langle,\rangle$ denote the pointwise inner product on $\La^{k}(M)$ given by $g$. 
    The global inner product is  defined by
	$$(\a,\b)=\int_{M}\langle\a,\b\rangle dVol_{g}.$$
	We also write $|\a|^{2}=\langle \a,\a\rangle$, $\|\a\|^{2}=\int_{M}|\a|^{2}dVol_{g}$, and let
	$$\La^{k}_{(2)}(M)=\{\a\in\La^{k}(M):\|\a\|^{2}<\infty \}.$$
	The operator of exterior differentiation is $d:\La^{k}_{0}(M)\rightarrow\La^{k+1}_{0}(M)$ and it satisfies $d^{2}=0$; its formal adjoint is $d^{\ast}:\La^{k+1}_{0}(M)\rightarrow\La^{k}_{0}(M)$; we have
	$$\forall\a\in\La^{k}_{0}(M),\ \forall\b\in\La^{k+1}_{0}(M),\ \int_{M}\langle d\a,\b\rangle=\int_{M}\langle\a,d^{\ast}\b\rangle.$$
	 We consider the space of $L^{2}$ closed forms
	 $$\mathcal{Z}_{(2)}^{k}(M)=\{\a\in \La^{k}_{(2)}(M): d\a=0 \},$$
	 where  it is understood that the equation $d\a=0$ holds weakly, that is to say
	 $$\forall\b\in\La^{k}_{0}(M),\ (\a,d^{\ast}\b)=0.$$
	 That is we have 
	 $$\mathcal{Z}_{(2)}^{k}(M)=\big{(}d^{\ast}(\La^{k+1}(M))\big{)}^{\bot}.$$
	 Define the space $\mathcal{B}^{k}_{(2)}(X)$ as follows:
	 $$\mathcal{B}^{k}_{(2)}(X)=\overline{\{du:u\in\La^{k-1}_{0}(X)\}}\subset\La^{k}_{(2)}(X).$$
	 Then, the $L^{2}$-reduced cohomology of $X$ is defined as
	 $$H^{k}_{(2)}(X)=\frac{\mathcal{Z}^{k}_{(2)}(X)}{\mathcal{B}^{k}_{(2)}(X)}.$$
	 We can also define
	 \begin{equation*}
	 \begin{split}
	 \mathcal{H}_{(2)}^{k}(M)&= (d^{\ast}(\La^{k+1}(M))^{\bot}\cap(d(\La^{k-1}(M)))^{\bot}\\
	 &=Z_{k}^{2}(M)\cap\{\a\in \La^{k}_{(2)}(M): d^{\ast}\a=0 \}\\
	 &=\{\a\in \La^{k}_{(2)}(M): d\a=d^{\ast}\a=0 \}.\\
	 \end{split}
	 \end{equation*}
	 Because the operator $d+d^{\ast}$ is elliptic, we have by elliptic regularity: $\mathcal{H}^{k}_{(2)}(M)\subset\La^{k}(M)$. The space $\La^{k}_{(2)}(M)$ has the following of Hodge-de Rham-Kodaira orthogonal decomposition
	 $$\La^{k}_{(2)}(M)=\mathcal{H}^{k}_{(2)}(M)\oplus\overline{d(\La^{k-1}_{0}(M))}\oplus\overline{d^{\ast}(\La^{k+1}_{0}(M))  },$$
	 where the closure is taken with respect to the $L^{2}$ topology. Therefore,
	 $$\mathcal{H}^{k}_{(2)}(X)\cong H^{k}_{(2)}(X).$$
\subsection{Riemannian manifolds with a parallel differential form}
	In this section, we recall some notations and definitions in differential geometry \cite{Ver}. Let $X$ be a smooth Riemannian  manifold. Given an odd or even from $\a\in\La^{\ast}(X)$, we denote by $\tilde{\a}$ its parity, which is equal to $0$ for even forms, and $1$ for odd forms. An operator $f\in\rm{End}(\La^{\ast}(X))$ preserving parity is called $even$, and one exchanging odd and even forms is odd.
	
	Given a $C^{\infty}$-linear map $\La^{1}(X)\xrightarrow{p}\La^{odd}(X)$ or  $\La^{1}(X)\xrightarrow{p}\La^{even}(X)$, $p$ can be uniquely extended to a $C^{\infty}$-linear derivation $\rho$ on $\La^{\ast}(X)$, using the rule
	\begin{equation*}
	\begin{split}
	&\rho|_{\La^{0}(X)}=0,\\
	& \rho|_{\La^{1}(X)}=p,\\
	&\rho(\a\wedge\b)=\rho(\a)\wedge\b+(-1)^{\tilde{\rho}\tilde{\a}}\a\wedge\rho(\b).\\
	\end{split}
	\end{equation*}
Then, $\rho$ is an even (or odd) differentiation of the graded commutative algebra $\La^{\ast}(X)$. Verbitsky gave a definition of the structure operator of $(X,\w)$ \cite[Definition 2.1]{Ver} .
	\begin{definition}\label{D2.1}
		Let $X$ be a Riemannian manifold equipped with a parallel differential $k$-form $\w$.\ Consider an operator $\underline{C}:\La^{1}(X)\rightarrow\La^{k-1}(X)$ mapping $\a\in\La^{1}(X)$ to $\ast(\ast\w\wedge\a)$. The corresponding derivation as above is
		$$C:\La^{\ast}(X)\rightarrow\La^{\ast+k-2}(X)$$
		is called the structure operator of $(X,\w)$. The parity of C is equal to that of $\w$.
	\end{definition}
	\begin{lemma}\label{L4}
		Let $X$ be a Riemannian manifold equipped with a parallel differential $k$-form $\w$, and  $L_{\w}$ the operator $\a\mapsto\a\wedge\w$. Then
		$$d_{C}:=L_{\w}d^{\ast}-(-1)^{\tilde{C}}d^{\ast}L_{\w}=\{L_{\w},d^{\ast}\},$$
		where $d_{C}$ is the supercommutator $\{d,C\}:=dC-(-1)^{\tilde{C}}Cd$.
	\end{lemma}
	We recall some Generalized K\"{a}hler identities which were proved by Verbitsky \cite[Proposition 2.5]{Ver} .
	\begin{proposition}\label{P1}
		Let $X$ be a Riemannian manifold equipped with a parallel differential $k$-form $\w$, $d_{C}$ the twisted de Rham operator constructed above,\ and $d^{\ast}_{C}$ its Hermitian adjoint. Then:\\
		(i) The following supercommutators vanish:
		$$\{d,d_{C}\}=0,\ \{d,d_{C}^{\ast}\}=0,\ \{d^{\ast},d_{C}\}=0,\ \{d^{\ast},d_{C}^{\ast}\}=0.$$
		(ii) The Laplacian $\De=\{d,d^{\ast}\}$ commutes with $L_{\w}:\a\mapsto\a\wedge\w$ and it adjoint operator, denoted as $\La_{\w}:\La^{i}(X)\rightarrow\La^{i-k}(X)$.
	\end{proposition}
	
	\begin{corollary}(\cite{Ver} Corollary 2.9)\label{C1}
		Let $(X,\w)$ be a Riemannian manifold equipped with a parallel differential $k$-form $\w$, and $\a$ a harmonic form on $X$. Then $\a\wedge\w$ is harmonic.
	\end{corollary}
\subsection{$G_{2}$-manifolds}\label{S1}
We begin with a crash course in $G_{2}$-geometry, touching upon the basic concepts and facts relevant for this article. For a more thorough and comprehensive discussion we refer to Joyce's book \cite{Joy2}.

Let $V$ be a $7$-dimensional vector space equipped with a non-degenerate $3$-form $\phi$.  Here by non-degenerate we mean that for each non-zero vector $v\in V$ the $2$-form $\rm{i}_{v}\phi$ on the quotient is $V/\langle v\rangle$ is symplectic. Then $V$ carries a unique inner product $g$ and orientation such that
$$i_{v_{1}}\phi\wedge i_{v_{2}}\phi\wedge\phi=6g(v_{1},v_{2})dvol, \forall v_{i}\in V.$$
An appropriate choice of basis identifies $\phi$ with the model 
$$\phi_{0}=dx^{123}+dx^{145}+dx^{167}+dx^{246}-dx^{257}-dx^{347}-dx^{356},$$
where $dx^{ijk}=dx^{i}\wedge dx^{j}\wedge dx^{k}$ and $\{x_{1},\ldots,x_{7} \}$ are standard coordinates on $\mathbb{R}^{7}$. The stabiliser of $\phi_{0}$ in $GL(\mathbb{R}^{7})$ is known to be isomorphic to the exceptional Lie group $G_{2}$.
\begin{definition}
	A $G_{2}$-manifold is a $7$-manifold $X$ equipped with a torsion-free $G_{2}$-structure $\phi$, that is $$\na_{g_{\phi}}\phi=0,$$
	where $g_{\phi}$ is the metric induce by $\phi$.
\end{definition}
Under the action of $G_{2}$, the space $\La^{2}(X)$ splits into irreducible representations, as follows:
\begin{equation*}
\La^{2}(X)=\La^{2}_{7}(X)\oplus\La_{14}^{2}(X),
\end{equation*}
where $\La^{i}_{j}$ is an irreducible $G_{2}$-representation of dimension $j$.
These summands can be characterized as follows:
\begin{equation}\nonumber
\begin{split}
&\La^{2}_{7}(X)=\{\a\in\La^{2}(X)\mid\ast(\a\wedge\phi)=2\a\}=\{\ast(u\wedge\ast\phi): u\in\La^{1}(X)\},\\
&\La^{2}_{14}(X)=\{\a\in\La^{2}(X)\mid\ast(\a\wedge\phi)=-\a\}=\{\a\in\La^{2}(X)\mid\a\wedge\ast\phi=0\}.\\
\end{split}
\end{equation}
We will  show that the map $L_{\phi}:\La^{p}\rightarrow\La^{p+2}$ on the complete $G_{2}$-manifold is injective for $p=0,1,2$ .
\begin{lemma}\label{L2}
	Let $(X,\phi)$ be a complete $G_{2}$-manifold. Then any $\a\in\La^{k}(X)$, $k=0,1,2$, satisfies the inequalities
	$$\|\a\|_{L^{2}(X)}\approx\|\a\wedge\phi\|_{L^{2}(X)}.$$
\end{lemma}
\begin{proof}
	Let $\a,\b\in\La^{0}(X)$,\ we observe that:
	$$(\a\wedge\phi)\wedge\ast(\b\wedge\phi)=7\a\b\ast1.$$
	We take $\b=\a$, then
	$$\|\a\|^{2}_{L^{2}(X)}=\frac{1}{7}\|\a\wedge\phi\|^{2}_{L^{2}(X)}.$$
	Let $\a,\b\in\La^{1}(X)$, we also observe that:
	$$\ast(\a\wedge\phi)\wedge(\b\wedge\phi)=4\ast\a\wedge\b,$$
	where we use the identity $\ast(\a\wedge\phi)\wedge\phi=-4\ast\a$, See \cite{Bry2}. We take $\b=\a$, then
	$$\|\a\|^{2}_{L^{2}(X)}=\frac{1}{4}\|\a\wedge\phi\|^{2}_{L^{2}(X)}.$$
	Let $\a\in\La^{2}(X)$, we can write $\a=\a^{7}+\a^{14}$, then  $\a\wedge\phi=2\ast\a^{7}-\ast\a^{14}$. Hence
	\begin{equation}\nonumber
	\|\a\wedge\phi\|^{2}_{L^{2}(X)}=4\|\a^{7}\|^{2}_{L^{2}(X)}+\|\a^{14}\|^{2}_{L^{2}(X)}\approx\|\a\|^{2}_{L^{2}(X)}.
	\end{equation}
\end{proof}
\subsection{$Spin(7)$-manifolds}\label{S2}
In this section we approach $Spin(7)$-geometry by thinking of the $4$-form $\Phi$, and not the metric, as the defining structure.
\begin{definition}
	A $4$-form $\Phi$ on an $8$-dimensional vector space $W$ is called \text{admissible} if there exists a basis of $W$ in which it is identified with the $4$-form $\Phi_{0}$ on $\mathbb{R}^{8}$ defined by
	\begin{equation*}
	\begin{split}
	\Phi_{0}&=dx^{1234}+dx^{1256}+dx^{1278}+dx^{1357}-dx^{1368}-dx^{1458}-dx^{1467}\\
	&-dx^{2358}-dx^{2367}-dx^{2457}+dx^{2468}+dx^{3456}+dx^{3478}+dx^{5678},\\
	\end{split}
	\end{equation*}
	where $dx^{ijkl}=dx^{i}\wedge dx^{j}\wedge dx^{k}\wedge dx^{l}$ and $\{x_{1},\ldots,x_{8}\}$ are standard coordinates on $\mathbb{R}^{8}$. The space of admissible forms on $W$ is denoted by $\mathscr{A}(W)$.
\end{definition}
A $Spin(7)$-structure on an $8$–dimensional manifold $X$ is an admissible 4–form $\Phi\in\Ga(\mathscr(TX))\subset\La^{4}(X)$. It follows that a manifold with  $Spin(7)$-structure is canonically equipped with a metric $g_{\Phi}$ and an orientation.
\begin{definition}
	A $Spin(7)$-manifold is a $8$-manifold $X$ equipped with a torsion-free $Spin(7)$-structure $\Phi$, that is $$\na_{g_{\Phi}}\Phi=0.$$
\end{definition}
Under the action of $Spin(7)$, the space $\La^{2}(X)$ splits into irreducible representations, as follows:
\begin{equation*}
\La^{2}(X)=\La^{2}_{7}(X)\oplus\La^{2}_{21}(X).
\end{equation*}
These summands can be characterized as follows:
\begin{equation}\nonumber
\begin{split}
&\La^{2}_{7}(X)=\{\a\in\La^{2}(X)\mid\ast(\a\wedge\Phi)=3\a\},\\
&\La^{2}_{21}(X)=\{\a\in\La^{2}(X)\mid\ast(\a\wedge\Phi)=-\a\}.\\
\end{split}
\end{equation}
We will also show that the map $L_{\Phi}:\La^{p}\rightarrow\La^{p+4}$ on the complete $Spin(7)$-manifold is injective for $p=0,1,2$.
\begin{lemma}\label{L3}
	Let $(X,\Phi)$ be a complete $Spin(7)$-manifold. Then any $\a\in\La^{k}(X)$, $k=0,1,2$, satisfies the inequalities
	$$\|\a\|_{L^{2}(X)}\approx\|\a\wedge\Phi\|_{L^{2}(X)}.$$
\end{lemma}
\begin{proof}
	Let $\a,\b\in\La^{0}(X)$,\ we observe that:
	$$(\a\wedge\Phi)\wedge\ast(\b\wedge\Phi)=14\a\b\ast1,$$
	then
	$$\|\a\|^{2}_{L^{2}(X)}=\frac{1}{14}\|\a\wedge\Phi\|^{2}_{L^{2}(X)}.$$
	Let $\a,\b\in\La^{1}(X)$, we also observe that:
	$$\ast(\a\wedge\Phi)\wedge(\b\wedge\Phi)=7\ast\a\wedge\b,$$
	where we use the identity $\ast(\a\wedge\Phi)\wedge\Phi=7\ast\a$, See \cite[Lemma 3.2]{Kar}. We take $\b=\a$, then
	$$\|\a\|^{2}_{L^{2}(X)}=\frac{1}{7}\|\a\wedge\Phi\|^{2}_{L^{2}(X)}.$$
	Let $\a\in\La^{2}(X)$, we write $\a=\a^{7}+\a^{21}$, then  $\a\wedge\Phi=3\ast\a^{7}-\ast\a^{21}$. Hence
	\begin{equation}\nonumber
	\|\a\wedge\Phi\|^{2}_{L^{2}(X)}=9\|\a^{7}\|^{2}_{L^{2}(X)}+\|\a^{21}\|^{2}_{L^{2}(X)}\approx\|\a\|^{2}_{L^{2}(X)}.
	\end{equation}
\end{proof}
\section{Vanishing theorems}
In this section, we will prove some vanishing theorems on $\mathcal{H}^{k}_{(2)}(X)$, Theorem \ref{T1}, \ref{T2} and \ref{T3}, along with some related results.
\subsection{A global  perturbation potential function}
We denote by $d_{C}$ is the twisted de Rham operator of $(X,\w)$. We then have following identity.
	\begin{proposition}\label{P3.1}
		\begin{equation}\label{E4.1}
		\mathcal{L}_{\na f}\w=(-1)^{k}dd_{C}f=-dd^{\ast}(f\w).
		\end{equation}
	\end{proposition}
	\begin{proof}
Since $\w$ is harmonic, the operator $d^{\w}=i_{\na f}\w$ can be expressed in the terms of the Hodge $d^{\ast}$-operator as $i_{\na f}\w=-d^{\ast}(f\w)$, See \cite[Remark 2.12]{HL}. We now give a detailed proof for the above identity. First noting that 
		\begin{equation*}
		\begin{split}
	i_{\na f}\w&=(-1)^{(n-k)(k-1)}\ast(df\wedge\ast\w)=(-1)^{(n-k)(k-1)}\ast d(f\wedge\ast\w)\\
	&=(-1)^{(n-k)(k-1)}\ast d\ast(f\w),\\
		\end{split}
		\end{equation*}
		and since $d^{\ast}=(-1)^{nk+n+1}\ast d\ast$, we conclude that $i_{\na f}\w=-d^{\ast}(f\w)$. We also observe that $d_{C}f=-(-1)^{k}d^{\ast}(f\w)$. Therefore we obtain the identity (\ref{E4.1}).
	\end{proof}
We can define the complete manifolds $(X,\w)$ which are given by a global perturbation potential function $f$.
\begin{definition}\label{D2}
Let $(X,\w)$ be a complete manifold equipped with a non-zero parallel differential $k$-form $\w$. If there is a function $f\in C^{2}(X)$ such that $$\w':=\w-\mathcal{L}_{\na f}\w.$$
is sufficiently small in $L^{\infty}$-norm, we call $(X,\w)$ a complete manifold given by a global perturbation potential.
\end{definition}
\begin{proposition}\label{P4}
	Suppose that the structure form $\w$ on a complete $G_{2}$- (or $Spin(7)$-) manifold is $\w=(-1)^{\tilde{C}}dd_{C}f+\w'$. Then there exists a positive constant $\de\in(0,1]$ with following significance. If $|\w'|\leq\de$, then
	$$-d^{\ast}df\geq C',$$
	where $C'$ is a uniform positive constant.
\end{proposition}
\begin{proof}
	First, we observer that $\w=-dd^{\ast}(f\w)+\w'=(-1)^{nk+n}d\ast(df\wedge\ast\w)+\w'$.
	
	By the hypothesis of $G_{2}$-manifold, $(n,k)=(7,3)$. Then the $G_{2}$-structure form $\phi$ satisfies
	\begin{equation*}
	\begin{split}
	7&=\ast(\phi\wedge\ast\phi)\\
	&=\ast(( d\ast (df\wedge\ast\phi)+\w')\wedge\ast\phi)\\
	&=\ast d(\ast(df\wedge\ast\phi)\wedge\ast\phi)+\ast(\w'\wedge\ast\phi)\\
	&=\ast d\ast (3df)+\ast(\w'\wedge\ast\phi)\\
	&=-3d^{\ast}df+\ast(\w'\wedge\ast\phi).\\
	\end{split}
	\end{equation*}
	Here we use the identity $\ast(\a\wedge\ast\phi)\wedge\ast\phi=3\ast\a$ for $\a\in\La^{1}(X)$, See \cite{Bry2} (3.4).
	
	By the hypothesis of $Spin(7)$-manifold, $(n,k)=(8,4)$. Then the $Spin(7)$-structure form $\Phi$ satisfies $\Phi=\ast\Phi$ and
	\begin{equation*}
	\begin{split}
	14&=\ast(\Phi\wedge\Phi)\\
	&=\ast(( d\ast (df\wedge\Phi)+\w')\wedge\Phi)\\
	&=\ast d(\ast(df\wedge\Phi))+\w'\wedge\Phi)\\
	&=\ast d\ast(7df)+\ast(\w'\wedge\Phi)\\
	&=-7d^{\ast}df+\ast(\w'\wedge\Phi).\\
	\end{split}
	\end{equation*}
	Here we use the identity $\ast(\a\wedge\Phi)\wedge\Phi)=7\ast\a$ for $\a\in\La^{1}(X)$. Therefore, in all cases, we get
	\begin{equation*}
	\begin{split}
	-d^{\ast}df&\geq C_{1}-C_{2}\ast(\w'\wedge\ast\w)\\
	&\geq C_{1}-C_{2}|\w'|\cdot|\w|\\
	&\geq C_{1}-C_{3}\de,\\
	\end{split}
	\end{equation*}
where $C_{1},C_{2},C_{3}$ are positive constants. We can choose $\de$ small enough to ensure that $C_{1}-C_{3}\de>0$.
\end{proof}
McNeal \cite{McN1} defined a class of complete K\"{a}hler manifolds which he called K\"{a}hler convex. We extend this to any Riemannian manifold with a non-zero parallel differential form. 
\begin{definition}\label{D1}
	Let $f\in C^{2}(X)$ be a function on $X$, $f\geq1$. We say that  $f$ dominates its gradient, or $f$ dominates $df$, if there exist constants $A>0$ and $B\geq0$ such that
	\begin{equation}\label{E29}
	|df|^{2}(x)\leq A+Bf(x),\ \forall x\in X.
	\end{equation}
\end{definition}	
Suppose that $B=0$, following the idea of Gromov \cite{Gro}, we can give a lower bound on the spectrum of the Laplace operator $\De$ on $\La^{(0)}_{(2)}$. 
\begin{proposition}\label{P10}
Let $(X,\w)$ be a Riemannian $n$-manifold equipped with a parallel non-zero differential $k$-form $\w$. Suppose that $\w=(-1)^{\tilde{C}}dd_{C}f+\w'$. If $|df|^{2}\leq A$, for some $A>0$, then any $\a\in\La^{0}_{(2)}(X)$ satisfies the inequality
$$
(C_{1}-C_{2}\|\w'\|_{L^{\infty}(X)})\|\a\|^{2}_{L^{2}(X)}
\leq A\langle\De\a,\a\rangle_{L^{2}(X)}.$$
where $C_{1}$ and $C_{2}$ are positive constants depending only on $g,n$.
\end{proposition}
\begin{proof}
	Since $\w$ is a parallel differential form, then $\na|\w|^{2}=0$, i.e. $|\w|=constant$. Letting $u\in\La^{0}(X)$, we observe that:
$$|u\w|^{2}=\ast\big{(}(u\w)\wedge(u\ast\w)\big{)}=|u|^{2}|\w|^{2}=constant|u|^{2},$$	
	and
	$$\De(u\w)\wedge\ast(u\w)=((\De u)\w)\wedge(u\ast\w)=constant(\De u\wedge\ast u).$$
	These imply that
	$$\|u\|_{L^{2}(X)}=constant\|u\w\|_{L^{2}(X)},\ \langle\De(u\w),u\w\rangle_{L^{2}(X)}=constant\langle\De u,u\rangle_{L^{2}(X)}.$$
	Now, we write $\b=\w\wedge\a=d\eta+\tilde{\a}$, for $\eta= (-1)^{\tilde{C}}d_{C}f\wedge\a$ and $\tilde{\a}=d_{C}f\wedge d\a+\w'\wedge\a$ and observe that
	$$\|\eta\|_{L^{2}(X)}\lesssim\|d_{C}f\|_{L^{\infty}(X)}\|\a\|_{L^{2}(X)}\lesssim A\|\a\|_{L^{2}(X)},$$
	and
	$$\|d^{\ast}v\|_{L^{2}(X)}\leq\langle\De v,v\rangle_{L^{2}(X)}^{1/2},\ \forall v\in\Om^{\bullet}(X).$$
	Next, since
	\begin{equation}\nonumber
	\begin{split}
	\|\tilde{\a}\|_{L^{2}(X)}&\lesssim\|d\a\|_{L^{2}(X)}\|d_{C}f\|_{L^{\infty}(X)}+\|\w'\|_{L^{\infty}(X)}\|\a\|_{L^{2}(X)}\\
	&\lesssim A\langle\De\a,\a\rangle_{L^{2}(X)}^{1/2}+\|\w'\|_{L^{\infty}(X)}\|\a\|_{L^{2}(X)}\\
	&\lesssim A\langle\De\b,\b\rangle_{L^{2}(X)}^{1/2}+\|\w'\|_{L^{\infty}(X)}\|\a\|_{L^{2}(X)},\\
	\end{split}
	\end{equation}
	we have	
	\begin{equation}\nonumber
	\begin{split}
		\|\b\|^{2}_{L^{2}(X)}&\leq|\langle\b,d\eta\rangle_{L^{2}(X)}|+|\langle\b,\tilde{\a}\rangle_{L^{2}(X)}|\\
	&\leq|\langle d^{\ast}\b,\eta\rangle_{L^{2}(X)}|+|\langle\b,\tilde{\a}\rangle_{L^{2}(X)}|\\
	&\leq\|d^{\ast}\b\|_{L^{2}(X)}\|\eta\|_{L^{2}(X)}+\|\b\|_{L^{2}(X)}\|\tilde{\a}\|_{L^{2}(X)}\\
&\lesssim A\langle\De\b,\b\rangle_{L^{2}(X)}^{1/2}\|\a\|_{L^{2}(X)}+ A\langle\De\b,\b\rangle_{L^{2}(X)}^{1/2}\|\b\|_{L^{2}(X)}
	+\|\w'\|_{L^{\infty}(X)}\|\a\|_{L^{2}(X)}\|\b\|_{L^{2}(X)}\\
	&\lesssim A\langle\De\a,\a\rangle_{L^{2}(X)}^{1/2}\|\b\|_{L^{2}(X)}+\|\w'\|_{L^{\infty}(X)}\|\b\|^{2}_{L^{2}(X)}.\\ 
\end{split}
\end{equation}

	This yields the desired estimate
	$$
	(C_{1}-C_{2}\|\w'\|_{L^{\infty}(X)})\|\a\|^{2}_{L^{2}(X)}
	\leq A\langle\De\a,\a\rangle_{L^{2}(X)}.$$
where $C_{1}, C_{2}$ are positive constants depending only on $g,n$.
\end{proof}
Suppose that $\w'$ is small enough in $L^{\infty}$. Then following Proposition \ref{P10}, the first eigenvalue of the Laplace operator $\De$ is nonzero. In \cite{CY}, Cheng and Yau proved that the first eigenvalue of $\De$ is zero on a complete Ricci-flat manifold. We then have
\begin{proposition}\label{P6}
Let $(X,\w)$ be a complete $G_{2}$- or $Spin(7)$-manifold. Suppose that $\w=(-1)^{\tilde{C}}dd_{C}f+\w'$. Also assume that the function $f$ satisfies the convexity condition on $X$, i.e., for some $A,B\geq 0$, $|df|^{2}\leq A+Bf$. If $\w'$ is small enough in $L^{\infty}$, then $B>0$.
\end{proposition}
\subsection{Vanishing theorems}
The main result of this subsection is a vanishing theorem for $\mathcal{H}^{k}_{(2)}(X)$, under the additional condition that $\w'$ is small enough.

Recall that a function $f$ is an exhaustion function on $X$ if
$$X_{k}=:\{x\in X: f(x)<k \}\subset X,\ \forall k\in\mathbb{R}$$
has compact closure.
\begin{proof}[\textbf{Proof of Theorem \ref{T1}}]
Let $\chi:\mathbb{R}\rightarrow\mathbb{R}$ be smooth, $0\leq\chi\leq1$ with
	$$\chi(x)=\left\{
	\begin{aligned}
	1&   &x\geq1, \\
	0&   &x\leq0,\\
	\end{aligned}
	\right.$$
	and define, for $k\in\mathbb{N}^{+}$,
	$$\psi_{k}(x)=\chi(k-f(x)).$$
	Note that $supp \psi_{k}\subset X_{k}$ and $\psi_{k}\equiv1$ on $X_{k-1}$.
	
	Suppose $h\in\mathcal{H}_{(2)}^{p}(X)$. Then by Corollary \ref{C1}, $\w\wedge h\in\mathcal{H}_{(2)}^{k+p}(X)$ and so it implies that $\w\wedge h$ is co-closed. Let $\textbf{h}=(-1)^{\tilde{C}}d_{C}f\wedge h$. Since $\psi_{k}\cdot\textbf{h}$ has compact support, an integration by parts gives
	\begin{equation}\label{E4}
	(\w\wedge h, d(\psi_{k}\cdot\textbf{h}))=(d^{\ast} (\w\wedge h), \psi_{k}\cdot\textbf{h})=0.
	\end{equation}
	Since $\w=(-1)^{\tilde{C}}dd_{C}f+\w'$ and $dh=0$ on $X$, we have
	\begin{equation}\label{E5}
	\begin{split}
	d(\psi_{k}\cdot\textbf{h})&=-\chi'(k-f)\cdot df\wedge d_{C}f\wedge h+\psi_{k}\cdot(\w-\w')\wedge h,\\
	\end{split}
	\end{equation}
We now substitute (\ref{E5}) into (\ref{E4}) and consider the two terms coming from the right-hand side of (\ref{E5}) separately. For the first term, the Cauchy-Schwarz inequality and the fact that $\w$ is bounded in the $\langle,\rangle$ inner product imply
	\begin{equation}\label{E7}
	\begin{split}
	|(\w\wedge h, -\chi'\cdot df\wedge d_{C}f\wedge h)|&\lesssim\int_{X_{k}\backslash X_{k-1}}|df\wedge d_{C}f|\cdot|h|^{2}\\
	&\lesssim\int_{X_{k}\backslash X_{k-1}}|df|^{2}\cdot|h|^{2}\\
	&\lesssim\int_{X_{k}\backslash X_{k-1}}(A+Bf)|h|^{2}\\
	&\lesssim (A+Bk)\int_{X_{k}\backslash X_{k-1}}|h|^{2},\\
	\end{split}
	\end{equation}
	for constants independent of $k$ and $A,B$ as in Definition \ref{D1}. The third inequality follows from our hypothesis on $df$.
	
 We claim that the assumption that $h\in\mathcal{H}^{p}_{(2)}(X)$ implies that there exists a subsequence $\{k_{l}\}$ such that
	\begin{equation}\label{E8}
	k_{l}\int_{X_{l_{k}}\backslash X_{l_{k-1}}}|h|^{2}\rightarrow0\ as\ l\rightarrow\infty.
	\end{equation}
	Otherwise, for some $c>0$,
	\begin{equation*}
	\begin{split}
	\int_{X}|h|^{2}&=\sum_{k=1}^{\infty}\int_{X_{k}\backslash X_{k-1}}|h|^{2}\\
	&\geq c\sum_{k=1}^{\infty}\frac{1}{k}\\
	&=\infty,\\
	\end{split}
	\end{equation*}
	a contradiction.
	
	For the term coming from the second term on the right-hand side for (\ref{E5}), 
	\begin{equation}\label{E9}
	\lim_{k\rightarrow\infty}(\w\wedge h,\psi_{k}\cdot(\w-\w')\wedge h)=\|\w\wedge h\|^{2}-(\w\wedge h,\w'\wedge h).
	\end{equation} 
	Substituting (\ref{E7})--(\ref{E9}) into (\ref{E4}), it follows that 
	\begin{equation}\label{E3.8}
	\|\w\wedge h\|^{2}_{L^{2}(X)}=(\w\wedge h,\w'\wedge h)\leq \|\w\wedge h\|_{L^{2}(X)}\|\w'\wedge h\|_{L^{2}(K)}.
	\end{equation}
Therefore, we complete this proof.
\end{proof}
\begin{proof}[\textbf{Proof of Corollary \ref{C2}}]
If $X$ is $G_{2}$ or $Spin(7)$-manifold, following Lemma \ref{L2}, \ref{L3}, then for $k=0,1,2$,
$$\|\a\|^{2}\approx\|\a\wedge\w\|^{2}, \forall \a\in\Om^{k}(X).$$
Following Theorem \ref{T1}, for any $L^{2}$-harmonic $2$-form $\a$, we then have
$$\|\a\|_{L^{2}(X)}\lesssim \|\w'\|_{L^{\infty}(X)}\|\a\|_{L^{2}(X)}\leq C\varepsilon\|\a\|_{L^{2}(X)},$$
where $C$ is a positive constant depending only on $n$. We can choose $\varepsilon$ small enough to ensure that $C\varepsilon<1$. Hence $\a=0$.
\end{proof}
\begin{lemma}\label{L1}
	Let  $(X,\w)$ be a complete Riemannian manifold equipped with a non-zero  parallel differential $k$-form $\w$.  If $\w':=d\theta$ is a $d$(sublinear)  $k$-form, then for any $h\in\mathcal{H}^{p}_{(2)}(X)$,  we have
	\begin{equation*}
	\langle\w\wedge h,\w'\wedge h\rangle_{L^{2}(X)}=0.
	\end{equation*}
\end{lemma}
\begin{proof}
	Let $\eta:\mathbb{R}\rightarrow\mathbb{R}$ be smooth, $0\leq\eta\leq1$,
	$$
	\eta(t)=\left\{
	\begin{aligned}
	1, &  & t\leq0 \\
	0,  &  & t\geq1
	\end{aligned}
	\right.
	$$
	and consider the compactly supported function
	$$f_{j}(x)=\eta(\rho(x_{0},x)-j),$$
	where $j$ is a positive integer.
	
	Let $h$ be a harmonic $p$-form in $L^{2}$. Observing that $d^{\ast}(\w\wedge h)=0$ since $\w\wedge h\in\mathcal{H}_{(2)}^{p+k}(X)$  and noticing that $f_{j}(\theta\wedge h)$ has compact support, one has
	\begin{equation}\label{E3}
	\begin{split}
	0&=(d^{\ast}(\w\wedge h),f_{j}(\theta\wedge h))\\
	&=(\w\wedge h,d(f_{j}\theta\wedge h))\\
	&=(\w\wedge h, f_{j}\w'\wedge h)+(\w\wedge h, df_{j}\wedge\theta\wedge h).\\
	\end{split}
	\end{equation}
	Since $0\leq f_{j}\leq 1$ and $\lim_{j\rightarrow\infty}f_{j}(x)(\w\wedge h)(x)=(\w\wedge h)(x)$, it follows from the dominated convergence theorem that
	\begin{equation}\label{E6}
	\lim_{j\rightarrow\infty}(\w\wedge h, f_{j}\w'\wedge h)\rangle=(\w\wedge h,\w'\wedge h).
	\end{equation}
	Following the idea in Theorem \ref{T1}, we  can also prove that there exists a subsequence $\{j_{i}\}_{i\geq1}$ such that
	\begin{equation}\label{E1}
	\lim_{i\rightarrow\infty}(j_{i}+1)\int_{B_{j_{i}+1}\backslash B_{j_{i}}}|h(x)|^{2}dx=0.
	\end{equation}
	Using  (\ref{E1}), one obtains
	\begin{equation}\label{E10}
	\lim_{i\rightarrow\infty}(\w\wedge h, df_{j}\wedge\theta\wedge h)=0
	\end{equation}
	It now follows from (\ref{E3}), (\ref{E6}) and (\ref{E10}) that $(\w\wedge h,\w'\wedge h)=0.$
\end{proof}
\begin{proof}[\textbf{Proof of Theorem \ref{T2}}]
The conclusion follows from Lemma \ref{L1} and Equation (\ref{E3.8}).
\end{proof}
\subsection{The $L^{2}$ estimates}
	\begin{proposition}\label{P3}
	Let $X$ be a complete Riemannian manifold, $\dim X=n$. Suppose that there is a function $f\in \La^{0}(X)$, $f\geq1$ such that 
	$$-\De f\geq C>0,\ |df|^{2}\leq A+Bf,\ B<C,$$
	where $A, B,C$ are positive constants. Then
	\begin{equation}\label{E45}
	m\int_{X}\frac{1}{f+M}|u|^{2}\leq \|du\|^{2},\ \forall u\in\La^{0}_{0}(X),
	\end{equation}
	where $M,m$ are positive constants depending on $A,B$. Furthermore, if $X$ is Ricci-flat, then
	\begin{equation*}
	m\int_{X}\frac{1}{f+M}|u|^{2}\leq
	\|du\|^{2}+\|d^{\ast}u\|^{2}, \ \forall u\in\La^{1}_{0}(X).
	\end{equation*} 
\end{proposition}
\begin{proof}
	If $\la$ is smooth function on $X$, we have an inequality
	$$\|du-ud\la\|^{2}=\|du\|^{2}+\|ud\la\|^{2}-(du^{2},d\la)\geq0.$$
	Thus
	\begin{equation}\label{E41}
	(u^{2},d^{\ast}d\la)\leq\|du\|^{2}+\|ud\la\|^{2}.
	\end{equation}
	Suppose now that $f$ dominates $df$. Replacing $f$ by $\tilde{f}=tf+1$, $t>0$ and small, we may assume\\
	(i) $\tilde{f}\geq1$, $x\in X$\\
	(ii) $|d\tilde{f}|^{2}\leq B\tilde{f}$, $x\in X$,\\
	where $B$ in (ii) above is the constant appearing in Definition \ref{D1}. Fix a $t$ such that (i) and (ii) hold. For notational convenience, we will continue to denote $\tilde{f}$ as just $f$, but unravel this abuse of notation at the end of the proof.
	
	For $\varepsilon>0$ to be determined, let $\la=-\varepsilon\log f$. Note that
	\begin{equation}\label{E4.4}
	\begin{split}
	d^{\ast}d\la&=-\frac{\varepsilon d^{\ast}df}{f}-\frac{\varepsilon\ast(\ast df\wedge df)}{f^{2}}\\
	&\geq\frac{\varepsilon C}{f}-\frac{\varepsilon|df|^{2}}{f^{2}} \\
	&\geq\frac{\varepsilon(C-B)}{f}.\\
	\end{split}
	\end{equation}
	Hence, (\ref{E4.4}) implies that 
	\begin{equation}\label{E42}
	(u^{2},d^{\ast}d\la)\geq \int_{X}\frac{\varepsilon(C-B)}{f}|u|^{2}.
	\end{equation}
	Note also that 
	\begin{equation}\label{E43}
	|d\la|^{2}=\frac{\varepsilon^{2}}{f^{2}}|df|^{2}\leq\varepsilon^{2}\frac{B}{f}.
	\end{equation}
	Substituting (\ref{E42})--(\ref{E43}) into (\ref{E41}), we obtain 
	\begin{equation}\label{E44}
	\int_{X}\frac{\varepsilon(C-B)-\varepsilon^{2}B}{f}|u|^{2}\leq \|du\|^{2}.
	\end{equation}
	As $C-B>0$, choose $\varepsilon$ so that $C-B-\varepsilon B=\kappa >0$. It follows from (\ref{E44}) that (\ref{E45})  holds with $\tilde{f}$ in place of $f$ when $M=0$ and $m=\kappa\varepsilon$. Recalling that $\tilde{f}=tf+1$, it follows that (\ref{E44}) holds for $f$ with $m=\frac{\kappa\varepsilon}{t}$ and $M=\frac{1}{t}$, which completes the proof.
	
	Suppose that $X$ is Ricci-flat. We consider the form $u\in\La^{1}_{0}(X)$, then the Weitzenb\"{o}ck formula gives
	\begin{equation*}
	\|du\|^{2}+\|d^{\ast}u\|^{2}=\|\na u\|^{2}.
	\end{equation*}
	Following the Kato inequality $|\na|u||\leq|\na u|$ and (\ref{E45}), we have
	\begin{equation*}
	m\int_{X}\frac{1}{f+M}|u|^{2}\leq \|\na|u|\|^{2}\leq \|\na u\|^{2}\leq
	\|du\|^{2}+\|d^{\ast}u\|^{2}.
	\end{equation*}
	We complete this proof.
\end{proof}

\begin{lemma}\label{L5}
	Let $(X,\w)$ be a  complete $G_{2}$- (or $Spin(7)$-) manifold. If $u\in\La^{2}(X)$, we denote $u=u_{1}+u_{2}$, where $u_{i}\in\La^{2}_{i}(X)$, then $\De u_{i}\in\La^{2}_{i}(X)$. Furthermore, we have identity $$\langle\De u, u\rangle=\langle \De u_{1}, u_{1}\rangle+\langle \De u_{2},u_{2}\rangle.$$
\end{lemma}
\begin{proof}
	Let $u_{i}\in\La^{2}_{i}(X)$, i.e., $u_{i}\wedge\w=c_{i}\ast u_{i}$, where $c_{i}$ is constant, See Subsection \ref{S1}, \ref{S2} . Following Proposition \ref{P1}, the Laplacian $\De=\{d,d^{\ast}\}$ commutes with $L_{\w}$. Thus 
	$$\De u_{i}\wedge\w=\De(u_{i}\wedge\w)=\De\ast c_{i}u_{i}=\ast c_{i}\De u_{i},$$ 
	i.e., $\De u_{i}\in \La^{2}_{i}(X)$. 
\end{proof}
\begin{proof}[\textbf{Proof of Theorem \ref{T3}}]
First consider the $k=0,1$ cases.\\	
Following Proposition \ref{P4}, the function $f$ on $X$ satisfies
$$-d^{\ast}df\geq C>0\ and\  |df|^{2}\leq A+Bf.$$
Noticing that Ricci curvatures on $G_{2}$- and $Spin(7)$-manifold are flat. If $B<C$, then by Proposition \ref{P3}
\begin{equation}
m\int_{X}\frac{1}{f+M}|u|^{2}\leq (\|du\|^{2}+\|d^{\ast}u\|^{2}),\ \forall u\in\La^{k}_{0}(X),
\end{equation}	
Now consider the $k=2$ case.\\
Over a complete $G_{2}$- (or $Spin(7)$-) manifold, $u\in\La^{2}(X)$ is decomposed into 
	$u=u_{1}+u_{2}$, where $u_{1}\in\La^{2}_{7}(X)$, $u_{2}\in\La^{2}_{14}(X)$ or $u_{2}\in\La^{2}_{21}(X)$. Moreover, we have identities $\ast(u_{i}\wedge\w)=c_{i}\w$, where $c_{1}, c_{2}$ are constants.
	
 Suppose now that $f$ dominates $df$. Replace $f$ by $\tilde{f}=tf+1$, $t>0$. Fix a $t$ such that the conditions (i) and (ii) in the proof of the Proposition \ref{P3} hold. For notational convenience, we will continue to denote $\tilde{f}$ as just $f$.
 
	We denote $\textbf{u}_{i}=u_{i}f^{-\frac{1}{2}}$.  Since $\textbf{u}_{i}$ has compact support,  an integration by parts gives
	\begin{equation}\label{E20}
	(\textbf{u}_{i}\wedge\w,d(\textbf{u}_{i}\wedge d_{C}f))=(d^{\ast}(\textbf{u}_{i}\wedge\w),\textbf{u}_{i}\wedge d_{C}f).
	\end{equation} 
	Since $\w=(-1)^{\tilde{C}}dd_{C}f+\w'$, we get
	\begin{equation}\label{E21}
	d(\textbf{u}_{i}\wedge d_{C}f))=d\textbf{u}_{i}\wedge d_{C}f+(-1)^{\tilde{C}}\textbf{u}_{i}\wedge(\w-\w').
	\end{equation}
	Note that $d^{\ast}(\mathbf{u}_{i}\wedge\w)=-c_{i}\ast d\mathbf{u}_{i}$. We now substitute (\ref{E21}) into (\ref{E20}), it gives that
	\begin{equation}\label{E22}
	\begin{split}
	(-1)^{\tilde{C}}(\textbf{u}_{i}\wedge\w,\textbf{u}_{i}\wedge\w)&=-(\textbf{u}_{i}\wedge\w,d\textbf{u}_{i}\wedge d_{C}f)+(d^{\ast}(\textbf{u}_{i}\wedge\w),\textbf{u}_{i}\wedge d_{C}f)+(-1)^{\tilde{C}}(\textbf{u}_{i}\wedge\w,\textbf{u}_{i}\wedge\w').\\
	&=-(c_{i}\ast\textbf{u}_{i},d\textbf{u}_{i}\wedge d_{C}f)-c_{i}(\ast d\textbf{u}_{i},\textbf{u}_{i}\wedge d_{C}f)+(-1)^{\tilde{C}}(\textbf{u}_{i}\wedge\w,\textbf{u}_{i}\wedge\w')\\
	&=I_{1}+I_{2}\\
	\end{split}
	\end{equation}
	Note that $$|d_{C}f|=|df\wedge\ast\w|\lesssim|df|.$$ 
	For the first and second terms coming from on the right-hand side of (\ref{E22}), the Cauchy-Schwarz inequality implies
	\begin{equation}\label{E25}
	\begin{split}
	|I_{1}|&=|c_{i}(\ast d\textbf{u}_{i},\textbf{u}_{i}\wedge d_{C}f)+(c_{i}\ast\textbf{u}_{i},d\textbf{u}_{i}\wedge d_{C}f)|\\
	&\lesssim |\int_{X}\textbf{u}_{i}\wedge d\textbf{u}_{i}\wedge d_{C}f|\\
	&=|\int_{X}\textbf{u}_{i}\wedge(f^{-\frac{1}{2}}du_{i}-\frac{1}{2}f^{-\frac{3}{2}}u_{i}\wedge df)\wedge d_{C}f|\\
	&\lesssim\int_{X}f^{-1}|u_{i}||du_{i}||df|+\int_{X}f^{-2}|u_{i}|^{2}|df|^{2}\\
	&\lesssim\int_{X}|du_{i}|^{2}+\int_{X}f^{-2}|u_{i}|^{2}|df|^{2}\\
	&\lesssim\int_{X}|du_{i}|^{2}+B\int_{X}f^{-1}|u_{i}|^{2},\\
	\end{split}
	\end{equation}
	for constants independent of $A,B$ as in Definition \ref{D1}.\\
For the third term coming from on the right-hand side of (\ref{E22}), we get
	\begin{equation}
	|I_{2}|=|(\textbf{u}_{i}\wedge\w,\textbf{u}_{i}\wedge\w')|\lesssim\|\w'\|_{L^{\infty}(X)}\int_{X}f^{-1}|u_{i}|^{2}.
	\end{equation}
	For the term coming from on the left-hand side of (\ref{E22}), we have
	\begin{equation}\label{E26}
	(\textbf{u}_{i}\wedge\w,\textbf{u}_{i}\wedge\w)=c_{i}^{2}\int_{X}\frac{|u_{i}|^{2}}{f}.
	\end{equation}
	Substituting (\ref{E25})--(\ref{E26}) into (\ref{E22}), it follows that
	\begin{equation}\label{E27}
	\int_{X}\frac{|u_{i}|^{2}}{f}\leq C\|du_{i}\|^{2}+C(B+\|\w'\|_{L^{\infty}(X)})\int_{X}\frac{|u_{i}|^{2}}{f}
	\end{equation}
	where $C$ is a positive constant independent of $A,B$. Provided that $C(B+\|\w'\|_{L^{\infty}(X)})\leq\frac{1}{2}$, rearrangement gives 
	\begin{equation*}\label{E28}
	\begin{split}
	\int_{X}\frac{|u|^{2}}{f}&\leq (\int_{X}\frac{|u_{1}|^{2}}{f}+\int_{X}\frac{|u_{2}|^{2}}{f})\\
	&\leq 2C(\|du_{1}\|^{2}+\|du_{2}\|^{2})\\
	&\leq 2C(\|du\|^{2}+\|d^{\ast}u\|^{2})\\
	\end{split}
	\end{equation*}
	where we use the Lemma \ref{L5}.
\end{proof}

The inequalities (\ref{E60}) on differential forms have an important application in the following problem:\\
The $L^{2}$-existence theorem and $L^{2}$-estimate of the Cartan-De Rham equation 
$$dv=u$$
where $u\in L^{2}(\La^{k}(X))$ is a given $(k+1)$-form satisfying
$$du=0.$$
\begin{proposition}
	Assume the hypotheses of Theorem \ref{T3}. Suppose that $f$ dominates $df$ and that the constant $B$ in Definition \ref{D1} is small enough. Then for any $u\in\La^{k}(X)$ with $k=0,1,2$ such that (i) $du=0$ and (ii) $fu\in\La^{k}_{(2)}(X)$ there exists a solution to $dv=u$ which satisfies the estimate
	$$\|v\|^{2}\leq C\int_{X}|u|^{2}\cdot(f+M),$$
	where the positive constant $C$ depends only on $A,B$.
\end{proposition}  
\begin{proof}
	Note that $|u|^{2}\leq f|u|^{2}\leq f^{2}|u|^{2}$ since $f\geq1$. Hence 
	$$\int_{X}|u|^{2}\leq\int_{X}f|u|^{2}\leq\int_{X}f^{2}|u|^{2}.$$
	Our proof here use McNeal's argument in \cite{McN1} for the $\bar{\pa}$-equation. Let $N=\{\a\in\La^{k}_{(2)}(X): d\a=0 \}$ and $S=\{d^{\ast}\b :\b\in\La^{k}_{0}\cap N  \}$. On $S$ consider the linear functional 
	$$d^{\ast}\b \rightarrow (\b,u).$$
	Using (\ref{E60}), we obtain
	\begin{equation}\label{E61}
	\begin{split}
	|(\b,u)|&=\big{|}(\frac{1}{\sqrt{f+M}}\b,\sqrt{f+M}u)\big{|}\\
	&\leq\big{(}\int_{X}\frac{1}{f+M}|\b|^{2}\big{)}^{\frac{1}{2}}\cdot\big{(}\int_{X}(f+M)|u|^{2}\big{)}^{\frac{1}{2}}\\
	&\lesssim\|d^{\ast}\b\|\big{(}\int_{X}(f+M)|u|^{2}\big{)}^{\frac{1}{2}}.
	\end{split}
	\end{equation}
	Thus the functional is bounded on $S$. However we also have $(\b,u)=0$ if $\b\in S^{\bot}$ since $du=0$, so (\ref{E61}) actually holds for all $\b\in\La^{k}_{0}(X)$. Since $\La^{k}_{0}(X)$ is dense in $$Dom(d^{\ast}):=\{u\in\La^{k}_{(2)}(X) :d^{\ast}u\in\La^{k-1}_{(2)}(X) \}$$ in the norm $\|u\|^{2}+\|d^{\ast}u\|^{2}$,  (\ref{E61}) holds for all $\b\in Dom(d^{\ast})$. The Hahn-Banach theorem extends the functional to all of $\La^{k}_{(2)}(X)$ and then the Riesz representation theorem gives a $v\in\La^{k-1}_{(2)}(X)$ such that
	$$(d^{\ast}\b,v)=(\b,u), \forall\b\in Dom(d^{\ast}).$$
	This is equivalent to $dv=u$, and 
	$$\|v\|\lesssim\big{(}\int_{X}|u|^{2}\cdot(f+M)\big{)}^{\frac{1}{2}},$$
	which is the claimed norm estimate.
\end{proof}

\section*{Acknowledgements}
We would like to thank the anonymous referee for  careful reading of my manuscript and helpful comments. This work is supported by Nature Science Foundation of China No. 11801539 and Postdoctoral Science Foundation of China No. 2017M621998, No. 2018T110616.

\end{document}